\documentclass[12pt, twoside, reqno]{article}
\usepackage{amsmath,amsthm}
\usepackage{amssymb}
\usepackage{enumitem}

\newtheorem{theorem}{Theorem}[section]
\newtheorem{corollary}[theorem]{Corollary}
\newtheorem{lemma}[theorem]{Lemma}

\newtheorem{problem}[theorem]{Problem}

\theoremstyle{definition}

\newtheorem{remark}[theorem]{Remark}

\newtheorem{conjecture}[theorem]{Conjecture}

\numberwithin{equation}{section}

\begin{document}

\baselineskip=17pt


\title{A conjecture of Erd\H os on $p+2^k$}

\author{Yong-Gao Chen\\
\small School of Mathematical Sciences and Institute of Mathematics\\
\small Nanjing Normal University,  Nanjing  210023,  P.R. China\\
\small ygchen@njnu.edu.cn}
\date{}

\maketitle


\begin{abstract} Let $\mathcal{U}$ be the set of positive odd integers that cannot
be represented as the sum of a prime and a power of two.  In this
paper, we prove that  $\mathcal{U}$ is not a union of finitely
many infinite arithmetic progressions and a set of asymptotic
density zero. This gives a negative answer to a conjecture of P.
Erd\H os. We pose several problems and a conjecture for further
research.
\end{abstract}

\renewcommand{\thefootnote}{}



\renewcommand{\thefootnote}{\arabic{footnote}}
\setcounter{footnote}{0}

\section{Introduction}

Let $\mathcal{P}$ be the set of all positive primes and
$\mathbb{N}$
 the set of all positive integers. In 1849, de Polignac
\cite{Polignac1849} conjectured that every odd integer greater
than $3$ can be represented as $p+2^k$,  where $p\in \mathcal{P}$
and $k\in \mathbb{N}$.  In 1934, Romanoff \cite{Romanoff1934}
 showed that there is a positive
proportion of positive odd integers that can be written as
$p+2^k$, where $p\in \mathcal{P}$ and $k\in \mathbb{N}$. In 2004,
Chen and Sun \cite{ChenSun2004} gave a quantitative version of
Romanoff's theorem. In 2018, Elsholtz and Schlage-Puchta
\cite{Elsholtz2018} showed that the proportion is at least
$0.107648$. In 1950, P. Erd\H os proved that there is an infinite
arithmetic progression of positive odd integers none of which can
be represented as the sum of a prime and a power of two. This
topic has brought many  subsequent works (Chen \cite{Chen2005},
Chen \cite{Chen2007}, Chen and Sun \cite{ChenSun2004}, Chen and Xu
\cite{ChenXu2024}, van der Corput \cite{Corput1950}, Crocker
\cite{Crocker1971}, Ding \cite{YDing2022}, Habsieger and Roblot
\cite{HabsiegerRoblot2006}, L\"u \cite{Lu2007}, Pan
\cite{Pan2011}, Pintz \cite{Pintz2006}, Sun \cite{XGSun2010}, Sun
\cite{Sun2000}, etc.).

Let $\mathcal{U}$ be the set of positive odd integers that cannot
be represented as $p+2^k$,  where $p\in \mathcal{P}$ and $k\in
\mathbb{N}$. A simple calculation shows that $\mathcal{U} =\{ 1
,3, 127, 149, 251, 331, \dots \}$. Erd\H os (see \cite{Bloom} or
\cite{Erdos1995}) posed the following conjecture:

{\bf Conjecture A} {\it The set $\mathcal{U}$ is the union of an
infinite arithmetic progression of positive odd integers and a set
of asymptotic density zero.}

 Erd\H os called this
conjecture ``rather silly".  This is Problem 16 of Bloom's list
\cite{Bloom} of Erd\H os' problems.  In this paper, we prove that
this conjecture of Erd\H os is false in the following stronger
form.

\begin{theorem}\label{thm1} For any
set $S$ of asymptotic density zero,  $\mathcal{U}\cup S$ is not a
union of finitely many infinite arithmetic progressions and a set
of asymptotic density zero.
\end{theorem}

From Theorem \ref{thm1}, by taking $S=\emptyset $ we have the
following result immediately.

\begin{corollary}\label{cor1} The set $\mathcal{U}$ is not a
union of finitely many infinite arithmetic progressions and a set
of asymptotic density zero.
\end{corollary}

Now we concern infinite arithmetic progressions $$\{ mh+a :
h=0,1,\dots \} \subseteq \mathcal{U}.$$ Firstly, we obtain the
following results.

\begin{theorem}\label{thm1b} We have $$\min m=11184810,\quad \min \omega (m)=7$$ and $\omega (m)=7$ if and
only if $m=11184810$, where the minimum is taken over all infinite
arithmetic progressions $\{ mh+a : h=0,1,\dots \}$ for which  $\{
mh+a : h=0,1,\dots \} \setminus \mathcal{U}$ has asymptotic
density zero and $\omega (m)$ denotes the number of distinct prime
divisors of $m$.
\end{theorem}

\begin{corollary}\label{cor2}  We have $$\min m=11184810,\quad \min \omega (m)=7$$ and $\omega (m)=7$ if and
only if $m=11184810$, where the minimum is taken over all infinite
arithmetic progressions $$\{ mh+a : h=0,1,\dots \} \subseteq
\mathcal{U}.$$ \end{corollary}

Very recently, Chen, Dai and Li \cite{ChenDaiLiarXiv} answered a
my problem affirmatively by proving that $\min m=11184810$, where
the minimum is taken over all infinite arithmetic progressions $\{
mh+a : h=0,1,\dots \} \subseteq \mathcal{U}$. Their proof depends
on a heavily calculation. They also gave some information on $\{
11184810h+a : h=0,1,\dots \} \subseteq \mathcal{U}$.\\

In this paper, we introduce the following notations. An infinite
arithmetic progression $\{ mh+a : h=0,1,\dots \}$ is called
\emph{quasi-non-representable} if $a>0$ and $\{ mh+a : h=0,1,\dots
\} \setminus \mathcal{U}$ has asymptotic density zero.  A
quasi-non-representable infinite arithmetic progression $\{ mh+a :
h=0,1,\dots \}$ is called  \emph{longest } if there is no
quasi-non-representable infinite arithmetic progression $\{ m'h+a'
: h=0,1,\dots \} $ such that $\{ mh+a : h=0,1,\dots \}$ is a
proper subset of $\{ m'h+a' : h=0,1,\dots \} $.

\begin{theorem}\label{thm1c} $\{ 11184810h+a : h=0,1,\dots \}$ is  a
longest quasi-non-representable infinite arithmetic progression if
and only if
\begin{eqnarray}\label{list} a\in &\{ & 509203, 762701,
992077, 1247173, 1254341, 1330207,  \nonumber\\
&& 1330319, 1730653, 1730681, 1976473, 2313487, 2344211, \nonumber\\
&& 2554843, 3177553, 3292241, 3419789, 3423373, 3661529, \nonumber\\
&& 3661543, 3784439, 4384979, 4442323, 4506097, 4507889, \nonumber\\
&&4626967, 5049251, 5050147, 6610811, 7117807, 7576559, \nonumber\\
&& 7629217, 8086751, 8101087,  8252819, 8253043,
8643209,\nonumber\\ && 9053711, 9053767, 9545351, 9560713,
9666029, 10219379, \nonumber\\ && 10280827, 10581097, 10609769,
10702091, 10913233,\nonumber\\ && 10913681\} .
\end{eqnarray}
\end{theorem}

\begin{corollary}\label{cor2a} Let $b$ be an integer. Then $$\{ 11184810h+b : h=0,1,\dots \} \subseteq
\mathcal{U}$$ if and only if $b\ge 0$ and $b\equiv
a\pmod{11184810}$ for some $a$ in the list \eqref{list}.
\end{corollary}

Now, I pose the following problems for further research:

\begin{problem}\label{prob0} Is there a constant $c$ such that if
$$\{ mh+a : h=0,1,\dots \} \subseteq \mathcal{U},$$
then the least odd prime divisor of $m$ is less than $c$?
\end{problem}

\begin{problem}\label{prob1} Let $A_i$ $(i\in I)$ be the collection of
all infinite arithmetic progressions of positive odd integers none
of which can be represented as the sum of a prime and a power of
two. Is it true that the set
$$\mathcal{U}\setminus \Big( \bigcup_{i\in I} A_i\Big)$$
has a positive upper (lower) asymptotic density? \end{problem}

 We claim that
$$\mathcal{U}\setminus \Big( \bigcup_{i\in I} A_i\Big)\not=\emptyset .$$
We will show that
\begin{equation}\label{w2a}\{ 1 ,3, 127\} \subseteq\mathcal{U}\setminus \Big( \bigcup_{i\in I} A_i\Big).\end{equation}
Let $a\in \{ 1 ,3, 127\} $. Then $a=2^k-1$ for some positive
integer $k$. Suppose that
$$a\notin  \mathcal{U}\setminus \Big( \bigcup_{i\in I} A_i\Big).$$
By $a\in \mathcal{U}$, there an $i\in I$ such that $a\in A_i$, say
$$a\in \{ mh+b: h=0,1,\dots \} \subseteq
\mathcal{U}.$$ It follows that
\begin{equation}\label{w2} \{ mh+a: h=0,1,\dots \} \subseteq
\mathcal{U}.\end{equation} Since $(m, a-2^k)=(m,-1)=1$, it follows
from Dirichlet's theorem on arithmetic progressions that there are
infinitely many primes $p\equiv a-2^k\pmod{m}$, i.e.  $p+2^k\equiv
a\pmod{m}$, a contradiction with \eqref{w2}.

Generally, we have the following result.

\begin{theorem}\label{thm2} Let the notations be as in Problem \ref{prob1} and let $a\in \mathcal{U}$.
Then  $a\in \cup_{i\in I} A_i$ if and only if there is an integer
$m>1$ such that $(a-2^k, m)>1$ for every positive integer $k$.
\end{theorem}

It is easy to see that Theorem \ref{thm2} is equivalent to the
following result.

\begin{corollary} Let the notations be as in Problem \ref{prob1}.
Then  $a\in \mathcal{U}\setminus \left( \cup_{i\in I} A_i\right)$
if and only if $a\in \mathcal{U}$ and $p(a-2^k)$ is unbound, where
$p(n)$ denotes the least prime divisor of $n$ and we appoint
$p(\pm 1)=+\infty$.
\end{corollary}

Now, I pose the following  problems for further research:

\begin{problem}\label{prob2a} Is there a positive odd integer $a$
such that $2^\ell-a$ is composite for all sufficiently large
integers $\ell $ but there are no integers $m>1$ with $(a-2^k,
m)>1$ for every positive integer $k$?
\end{problem}

\begin{problem}\label{prob2} What is the least integer in $\cup_{i\in I} A_i$?
That is, what is the least integer $e_1$ in $\mathcal{U}$ such
that there is an integer $m>1$ with $(e_1-2^k, m)>1$ for every
positive integer $k$?
\end{problem}

\begin{problem}\label{prob3} What is the least positive odd integer $e_2$  such that
there is an integer $m>1$ with  $(e_2-2^k, m)>1$ for every
positive integer $k$?
\end{problem}

If there are only finitely many Mersenne primes, then  $2^\ell-1$
is composite for all sufficiently large integers $\ell $, but
there are no integers $m>1$ with $(1-2^1, m)>1$. So Problem
\ref{prob2a} is affirmative under the assumption that there are
only finitely many Mersenne primes.   By \eqref{w2a} and
$\mathcal{U} =\{ 1 ,3, 127, 149, 251, 331, \dots \}$, we have
$e_1\ge 149$. Noting that $(1-2^1, m)=1$, $(3-2^1, m)=1$, $(5-2^2,
m)=1$, $(7-2^3, m)=1$ and $(9-2^3, m)=1$ for any integer $m$, we
know that $e_2\ge 11$. Currently, we cannot prove that $e_1\not=
149$ and $e_2\not= 11$. Is it true that $2^k-11$ is a prime for
infinitely many positive integers $k$? Is it true that $2^k-149$
is a prime for infinitely many positive integers $k$?

Since $(509203-2^k, 2^{24}-1)>1$ for $1\le k\le 24$, it follows
that $(509203-2^k, 2^{24}-1)>1$ for every positive integer $k$.
Noting that $509203-2^k$ are composite for $1\le k\le 18$ and
$2^{19}>509203$, we have $509203\in \mathcal{U}$. Hence, $e_2\le
e_1\le 509203$. Is it true that $e_2= e_1= 509203$?

Let $\mathcal{W}_1$ be the set of all odd positive integers $a$
for which there is an integer $m>1$ with  $(a-2^k, m)>1$ for every
positive integer $k$, and let $\mathcal{W}_2$ be the set of all
positive integers $a$ for which there is no an integer $m>1$ with
$(a-2^k, m)>1$ for every positive integer $k$. By Corollary
\ref{cor2a},
$$\{ 11184810h+509203 : h=0,1,\dots \} \subseteq \mathcal{U}.$$ It
follows from Theorem \ref{thm2} that $$\{ 11184810h+509203 :
h=0,1,\dots \} \subseteq \mathcal{W}_1.$$

We pose a conjecture here for further research.

\begin{conjecture} The set $\mathcal{W}_2$ has a positive lower asymptotic density.\end{conjecture}

Now we give more information on $\cup_{i\in I} A_i$. Let $J$ be
the set of indices $i$ that the common difference of $A_i$ is the
product of a power of $2$ and a squarefree odd integer and let $K$
be the set of indices $i$ that the  common difference of $A_i$ is
a squarefree integer. We have the following result.

\begin{theorem}\label{thm3} (i) We have  $$\bigcup_{i\in I} A_i=\bigcup_{i\in J}
A_i;$$

(ii) Assume that there are infinitely many Mersenne primes. Then
$$\bigcup_{i\in I} A_i=\bigcup_{i\in K} A_i.$$
\end{theorem}

From Theorem \ref{thm3} (ii) and the comments just after Problem
\ref{prob3}, we have the following result immediately.

\begin{corollary} Either Problem
\ref{prob2a} is affirmative or
$$\bigcup_{i\in I} A_i=\bigcup_{i\in K} A_i$$
or both.
\end{corollary}

In this paper, we define $a\pmod u=\{ uh+a : h\in \mathbb{Z} \} $.
For an odd prime $p$, let $r_p$ be the multiplicative order of $2$
modulo $p$.

This paper is organized as follows. In Section \ref{Thm1sec}, we
give a proof of Theorem \ref{thm1}. In Section
\ref{ConjectureAsec}, we present a different negative answer to
Conjecture A. Although Theorem \ref{thm1} implies  that Conjecture
A is false, we believe that the proof in this section is also
interesting enough. In Section \ref{thm1bsec}, we give  proofs of
Theorem \ref{thm1b}, Theorem \ref{thm1c} and their corollaries.
Basing on Theorem \ref{thm1b}, we also give a negative answer to
Conjecture A. In the final section, we give proofs of Theorems
\ref{thm2} and \ref{thm3}.

\section{Proof of Theorem \ref{thm1}}\label{Thm1sec}

 We will prove Theorem \ref{thm1} by  contradiction.  Suppose
that Theorem \ref{thm1} is false. Then there is a set $S$ of
asymptotic density zero such that $\mathcal{U}\cup S$ is a union
of finitely many infinite arithmetic progressions and a set of
asymptotic density zero.    That is, $\mathcal{U}\cup S$ can be
written as
\begin{equation}\label{g1}\mathcal{U}\cup S=\bigcup_{i=1}^t\{ m_ih+a_i : h=0,1,\dots \} \cup
W,\end{equation} where $m_i, a_i$ $(1\le i\le t)$ are  positive
integers and $W$ has asymptotic density zero. It is clear that
every $m_i$ is even, otherwise, for any integer $h$, one of
$m_ih+a_i$ and $m_i(h+1)+a_i$ is even, a contradiction with $\{
m_ih+a_i : h=0,1,\dots \} \subseteq \mathcal{U}\cup S$. Since
there is an infinite arithmetic progression of positive odd
integers none of which can be represented as the sum of a prime
and a power of two (see \cite{Erdos1950}), we have $t\ge 1$.

We will find distinct primes $p_1, \dots , p_{\ell +3}, \dots ,
p_s$ and positive integers $\alpha, a, c$ such that

(i) $m_1\cdots m_t \mid (p_1\cdots p_s)^\alpha $;

(ii) $p_{\ell +3}\nmid m_i$ for all $1\le i\le t$;

(iii) for any positive integer $k$, $a-2^k$ can be divided by at
least one of $p_1,\dots , p_{\ell +3}$.

(iv) $p_i \nmid a-2^c$ for all $1\le i\le s$ with $i\not= \ell
+3$.\\

Suppose that we have found distinct primes $p_1, \dots , p_{\ell
+3}, \dots , p_s$ and positive integers $\alpha, a, c$ satisfying
(i)-(iv).

By (i), one of $p_1, \dots , p_{\ell +3}, \dots , p_s$ is $2$. In
view of (ii), $p_{\ell +3}\not= 2$. It follows from (iv) that
$2\nmid a-2^c$. So $a$ is odd. Let
$$n\in \{ (p_1\cdots
p_s)^\alpha h+a : h=0,1,\dots \} .$$ Then $n$ is odd. We will show
that
\begin{equation}\label{gg1}n\in \mathcal{U}\cup \{ 2^k + p_i : k=1,2,\dots ; 1\le
i\le \ell +3 \}.\end{equation} If $n\in \mathcal{U}$, then we are
done. Suppose that $n\notin \mathcal{U}$. Then $n$ can be
represented as $n=p+2^k$, $p\in \mathcal{P}$ and $k\in
\mathbb{N}$. By (iii), there exists $1\le i\le \ell +3$ such that
$p_i\mid a-2^k$. It follows that $p=n-2^k\equiv a-2^k\equiv
0\pmod{p_i}$. So $p=p_i$. Thus,
$$n\in \{ 2^k + p_i : k=1,2,\dots ;
1\le i\le \ell +3 \} .$$ In all cases, \eqref{gg1} holds. Hence
\begin{eqnarray*}&& \{ (p_1\cdots p_s)^\alpha
h+a : h=0,1,\dots \}\\
&\subseteq & \mathcal{U}\cup \{ 2^k + p_i : k=1,2,\dots ; 1\le
i\le \ell +3 \}.\end{eqnarray*} By  \eqref{g1}, we have
\begin{eqnarray*}&& \{ (p_1\cdots p_s)^\alpha
h+a : h=0,1,\dots \}\cup S\\
& \subseteq & \mathcal{U}\cup S\cup \{ 2^k + p_i : k=1,2,\dots ;
1\le
i\le \ell +3 \}\\
&=& \bigcup_{i=1}^t\{ m_ih+a_i : h=0,1,\dots \} \cup W \\
&&  \cup \{ 2^k + p_i : k=1,2,\dots ; 1\le i\le \ell +3 \}
.\end{eqnarray*} Since
$$\{ (p_1\cdots
p_s)^\alpha h+a : h=0,1,\dots \}$$
 has asymptotic density $(p_1\cdots
p_s)^{-\alpha} $ and each of sets
$$W,\ S,\  \{ 2^k + p_i : k=1,2,\dots ; 1\le i\le \ell +3
\}$$ has asymptotic density $0$,  there exists an integer $h_0\ge
0$ with
$$ (p_1\cdots p_s)^\alpha h_0+a\in \bigcup_{i=1}^t\{ m_ih+a_i : h=0,1,\dots
\} .$$ Without loss of generality, we may assume that
$$ (p_1\cdots p_s)^\alpha h_0+a\in \{ m_1h+a_1 : h=0,1,\dots
\} .$$ So
$$(p_1\cdots p_s)^\alpha h_0+a\equiv a_1\pmod{m_1}.$$
It follows from (i) that \begin{equation}\label{f1}a\equiv
a_1\pmod{m_1}.\end{equation} By (i) and (ii), we have
\begin{equation}\label{f2}m_1\mid (p_1\cdots p_s/p_{\ell +3})^\alpha .\end{equation}
Since
$$( (p_1\cdots p_s/p_{\ell +3})^\alpha, p_{\ell +3}^\alpha)=1,$$
there exists a positive integer $a'$  such that
\begin{equation}\label{f3}a'\equiv a \pmod{(p_1\cdots p_s/p_{\ell +3})^\alpha},\quad
a'\equiv 2^c+1\pmod{p_{\ell +3}^\alpha}.\end{equation} Then
\begin{equation}\label{f4}p_{\ell +3}\nmid a'-2^c, \quad a'-2^c\equiv a-2^c\pmod{p_i^\alpha }, \quad 1\le i\le s, i\not= \ell +3.\end{equation}
By \eqref{f1}-\eqref{f3},
\begin{equation}\label{f5}a'\equiv a\equiv a_1 \pmod{m_1}.\end{equation}
By (iv) and \eqref{f4}, we have $p_i\nmid a'-2^c$ for all $1\le
i\le s$. For $p\in \mathcal{P}$ and $k\in \mathbb{N}$ with
$$p\equiv a'-2^c \pmod{(p_1\cdots p_s)^\alpha},\quad k\equiv
c\pmod{\varphi ((p_1\cdots p_s)^\alpha)},$$ by Euler's theorem we
have
\begin{equation}\label{f6}p+2^k\equiv a'-2^c +2^c\equiv a'\pmod{(p_1\cdots p_s)^\alpha}.\end{equation}
In view of (i), \eqref{f5} and \eqref{f6},
\begin{equation}\label{ff6}p+2^k\equiv a' \equiv a_1 \pmod{m_1}.\end{equation}
For any integer $n$, let $r(n)$ be the number of solutions of
$n=p+2^k$, $p\in \mathcal{P}$ and $k\in \mathbb{N}$. By
\eqref{ff6} and Dirichlet's theorem on arithmetic progressions, we
have
\begin{eqnarray}\label{ee4c}&&\sum_{\substack{n\le x\\ n\equiv a_1\hskip -3mm\pmod{m_1}}} r(n)\nonumber\\
&\ge & |\{ (p, k) : p+2^k\le x, p\in \mathcal{P}, p\equiv a'-2^c
\pmod{(p_1\cdots p_s)^\alpha},\nonumber\\
&& \hskip 3.5cm k\in \mathbb{N}, k\equiv
c\pmod{\varphi ((p_1\cdots p_s)^\alpha)} \} |\nonumber\\
&\ge & | \{ p\le \frac x2 : p\in \mathcal{P}, p\equiv a'-2^c
\pmod{(p_1\cdots p_s)^\alpha} \} |\nonumber\\
&& \cdot |\{ k\le \frac{\log (x/2)}{\log 2} : k\in \mathbb{N},
k\equiv
c\pmod{\varphi ((p_1\cdots p_s)^\alpha)} \} |\nonumber\\
&=& (1+o(1)) \frac x{2\varphi ((p_1\cdots p_s)^\alpha ) \log
(x/2)} \cdot \frac{\log (x/2)}{(\log 2)\varphi ((p_1\cdots
p_s)^\alpha) }\nonumber\\
&\gg & x,\end{eqnarray} where $\gg$ and the following $\ll$ are
Vinogradov's symbols. As an application of sieve methods, we have
(see \cite{ChenSun2004}
 or \cite{Romanoff1934})
\begin{equation}\label{ee4d}\sum_{n\le x} r(n)^2\ll x.\end{equation}
By the Cauchy-Schwarz inequality, we have
\begin{eqnarray*}&&\Big(\sum_{\substack{n\le x\\ n\equiv a_1\hskip -3mm\pmod{m_1}}}
r(n)\Big)^2\\
&=&\Big(\sum_{\substack{n\le x\\ n\equiv a_1\hskip
-3mm\pmod{m_1}\\ r(n)\ge 1}}
r(n)\Big)^2\\
&\le & \Big(\sum_{\substack{n\le x\\  n\equiv a_1\hskip -3mm\pmod{m_1}\\ r(n)\ge 1}} 1\Big) \Big(\sum_{\substack{n\le x\\
 n\equiv a_1\hskip -3mm\pmod{m_1}\\ r(n)\ge 1}}
r(n)^2\Big)\\
&\le & \Big(\sum_{\substack{n\le x\\  n\equiv a_1\hskip
-3mm\pmod{m_1}\\ r(n)\ge 1}} 1\Big) \Big(\sum_{n\le x}
r(n)^2\Big).
\end{eqnarray*}
It follows from \eqref{ee4c} and \eqref{ee4d} that
$$\sum_{\substack{n\le x\\  n\equiv a_1\hskip -3mm\pmod{m_1}\\ r(n)\ge 1}} 1\gg x.$$ This means that there is a
positive proportion integers in $$\{ m_1 h+a_1 : h=0,1,\dots \} $$
which can be represented as $n=p+2^k$, $p\in \mathcal{P}$ and
$k\in \mathbb{N}$, a contradiction with
$$\{ m_1 h+a_1 : h=0,1,\dots \} \subseteq \mathcal{U}\cup S.$$

The remaining thing is to find primes $p_1, \dots , p_{\ell +3},
\dots , p_s$ and positive integers $\alpha, a, c$ satisfying
(i)-(iv).

Let $\ell$ be an integer with $2^\ell > 4m_i$ for all $1\le i\le
t$. It is clear that $\ell >3$. For $1\le i\le \ell $, let $p_i$
be a prime factor of $2^{2^{i-1}}+1$ and let $c_i=2^{2^{i-1}-1}$.
For $\ell +1\le i\le \ell +3$, let $p_{i}$ be a prime factor of
$(2^{3\cdot 2^{2\ell -i}}+1)/(2^{2^{2\ell -i}}+1)$ and let
$c_{i}=2^{(i-\ell ) 2^{\ell } -1}$. Let $p_{j}$ $(\ell +4\le j\le
s)$ be all distinct prime factors of $m_1\cdots m_t$ which are not
$p_1,\dots , p_{\ell +3}$ and let $c_{j}=2^{3\cdot 2^{\ell }-1}+1$
$(\ell +4\le j\le s)$. Now we prove that $p_1, \dots , p_s$ are
distinct. Since $p_{j}$ $(\ell +4\le j\le s)$ are all distinct
prime factors of $m_1\cdots m_t$ which are not $p_1,\dots ,
p_{\ell +3}$, it is enough to prove that $p_1,\dots ,p_{\ell +3}$
are distinct. Recall that $r_p$ is  the multiplicative order of
$2$ modulo $p$, by $p_i\mid 2^{2^{i-1}}+1$ $(1\le i\le \ell )$ we
have $r_{p_i}=2^i$ $(1\le i\le \ell )$. For $\ell +1\le i\le
\ell+3$, by
$$p_i\mid \frac{2^{3\cdot 2^{2\ell -i}}+1}{2^{2^{2\ell -i}}+1},$$
we have
$$p_i\mid 2^{3\cdot 2^{2\ell -i}}+1.$$
It follows that
$$r_{p_i}\mid 3\cdot 2^{2\ell -i+1},\quad r_{p_i}\nmid 3\cdot 2^{2\ell
-i},\quad \ell +1\le i\le \ell +3.$$ So
$$r_{p_i}\in \{ 3\cdot 2^{2\ell -i+1}, 2^{2\ell -i+1} \} , \quad \ell +1\le i\le \ell +3.$$
Suppose that $r_{p_i}=2^{2\ell -i+1}$ for some $\ell +1\le i\le
\ell +3$. Then
$$2^{2^{2\ell -i}}\not\equiv
1\pmod{p_i}$$ and
$$2^{2^{2\ell -i+1}}\equiv 2^{r_{p_i}}\equiv 1\pmod{p_i}.$$
It follows that \begin{equation}\label{w1}2^{2^{2\ell -i}}\equiv
-1\pmod{p_i}.\end{equation} Thus,
$$\frac{2^{3\cdot 2^{2\ell -i}}+1}{2^{2^{2\ell -i}}+1}=2^{2\cdot 2^{2\ell
-i}}-2^{2^{2\ell -i}}+1\equiv (-1)^2-(-1)+1\equiv 3\pmod{p_i}.$$
Noting that
$$p_i\mid \frac{2^{3\cdot 2^{2\ell -i}}+1}{2^{2^{2\ell -i}}+1},$$
we have $p_i=3$. By $\ell >3$, we have $2\ell -i\ge \ell -3\ge 1$
and
$$2^{2^{2\ell -i}}\equiv (-1)^{2^{2\ell -i}}\equiv 1\pmod 3,$$
a contradiction with \eqref{w1}. Hence
$$r_{p_i}=3\cdot 2^{2\ell -i+1},\quad \ell +1\le i\le \ell+3.$$
Since $r_1, \dots , r_{\ell+3}$ are distinct, it follows that
$p_1,\dots ,p_{\ell +3}$ are distinct.

 We choose a positive integer
$\alpha$ such that $p_i^\alpha \nmid m_1\cdots m_t$ for all $1\le
i\le s$. Since all prime factors of $m_1\cdots m_t$ are in $\{
p_1, \dots , p_s\} $, it follows that $m_1\cdots m_t\mid
(p_1\cdots p_s)^\alpha$. That is, (i) holds. Noting that $r_{\ell
+3}$ is the multiplicative order of $2$ modulo $p_{\ell +3}$, we
have $r_{\ell +3} \mid p_{\ell +3}-1$.  In view of  $2^\ell >
4m_i$, we have
$$p_{\ell +3}\ge r_{\ell+3}+1=3\cdot 2^{2\ell -\ell -3+1}+1=3\cdot
2^{\ell -2}+1>m_{i},\quad 1\le i\le t.$$ Hence, (ii) holds.

By the Chinese remainder theorem, there exists a positive integer
$a$ such that \begin{equation}\label{q1}a\equiv
c_i\pmod{p_i^\alpha}, \quad i=1,2,\dots , s.\end{equation} Now we
prove (iii). Let $k$ be a positive integer. We will prove that
$a-2^k$ can be divided by at least one of $p_1,\dots , p_{\ell
+3}$.

If $2\mid k$, then by $p_1=3$ and $c_1=1$, we have
$$a-2^k\equiv c_1-2^k\equiv 1-1\equiv 0\pmod{p_1}.$$
It follows that $p_1\mid a-2^k$. Now we assume that $2\nmid k$.
That is, $k\equiv 1\pmod 2$. Let $j$ be the largest integer such
that
$$k\equiv 1+2+2^2+\cdots +2^{j-1}\pmod{2^{j}}.$$ Then $j\ge 1$.
Write
$$k=1+2+2^2+\cdots +2^{j-1}+2^ju=2^j-1+2^ju.$$
By the definition of $j$, $u$ is even. Let $u=2v$. Then
$$k=2^{j+1}v+2^j-1.$$
If $j\le \ell -1$, then by
\begin{eqnarray*}2^k&=&2^{2^{j+1}v+2^j-1}=(2^{2^{j+1}}-1+1)^v2^{2^j-1}\\
&=&((2^{2^{j}}-1)(2^{2^{j}}+1)+1)^vc_{j+1}\equiv
c_{j+1}\pmod{p_{j+1}},\end{eqnarray*} we have
$$a-2^k\equiv c_{j+1}-c_{j+1}\equiv 0\pmod{p_{j+1}}.$$
It follows that $p_{j+1}\mid a-2^k$. If $j\ge \ell$, then we can
write
\begin{eqnarray*}k&=&1+2+2^2+\cdots +2^{j-1}+2^ju\\
&=&1+2+2^2+\cdots +2^{\ell-1}+2^\ell u'\\
&=&2^\ell -1+2^\ell u'.\end{eqnarray*} Let $u'=3d+r$, $r\in \{
0,1,2\}$. Then
$$k=3\cdot 2^\ell d+ (r+1)2^\ell -1.$$
Since
$$p_{\ell +r+1}\mid 2^{3\cdot 2^{\ell -r-1}}+1,$$
it follows that
$$2^{3\cdot 2^\ell}\equiv 1\pmod{p_{\ell +r+1}}.$$
Thus,
$$2^k=2^{3\cdot 2^\ell d+ (r+1)2^\ell -1}\equiv 2^{(r+1)2^\ell -1}\pmod{p_{\ell +r+1}}.$$
So
$$a-2^k\equiv c_{\ell +r+1}-2^{(r+1)2^\ell -1}\equiv 0\pmod{p_{\ell +r+1}}.$$
It follows that $p_{\ell +r+1}\mid a-2^k$. In all cases, $a-2^k$
can be divided by at least one of $p_1,\dots , p_{\ell +3}$. That
is, (iii) holds.

Finally, let $c=3\cdot 2^{\ell }-1$. We will prove that (iv)
holds. That is, for all $1\le i\le s$ with $i\not= \ell +3$,
$p_i\nmid a-2^{c}$.

For $1\le i\le \ell $, by $p_i\mid 2^{2^{i-1}}+1$ we have $p_i>2$
and \begin{eqnarray*}2\left( a-2^{c}\right)&\equiv & 2a-2^{3\cdot
2^{\ell }}\equiv 2c_i-2^{3\cdot 2^{\ell }} \\
&\equiv & 2^{2^{i-1}}-2^{3\cdot 2^{\ell }}\equiv -1-(-1)^{3\cdot
2^{\ell -i+1}}\equiv -2\pmod{p_i}.\end{eqnarray*} It follows that
$$p_i\nmid a-2^{c}.$$

For $i=\ell +1, \ell +2$, by $p_{i}\mid 2^{3\cdot 2^{2\ell-i}}+1$
and $\ell>3$, we have $p_{i}>2$ and
\begin{equation}\label{ee1}2\left( a-2^{c}\right)\equiv 2c_{i}-2^{3\cdot 2^{\ell }}
\equiv 2^{(i-\ell )2^{\ell }}-2^{3\cdot 2^{\ell
}}\pmod{p_{i}}.\end{equation} Noting that
$$2^{3\cdot 2^{\ell
}}\equiv (2^{3\cdot 2^{2\ell-i}})^{2^{i-\ell}}\equiv
(-1)^{2^{i-\ell}}\equiv 1\pmod{p_{i}},$$ by \eqref{ee1} we have
\begin{equation}\label{ee1q}2\left( a-2^{c}\right)\equiv 2^{(i-\ell )2^{\ell
}}-1\pmod{p_{i}}.\end{equation} Since $r_{p_i}=3\cdot
2^{2\ell-i+1}$ cannot divide $(i-\ell )2^\ell $, it follows that
\begin{equation}\label{q2}2^{(i-\ell )2^{\ell
}}-1\not\equiv 0\pmod{p_{i}}.\end{equation} In view of $p_{i}>2$,
\eqref{ee1q} and \eqref{q2}, we know that $p_i\nmid a-2^{c}$.

Since $2\mid m_i$ $(1\le i\le t)$ and $p_1, \dots , p_{\ell +3}$
are odd,  one of $p_i$ $(\ell +4\le i\le s)$ is $2$. By
$c_i=2^{3\cdot 2^{\ell }-1}+1$ $(\ell +4\le i\le s)$ and
\eqref{q1}, we know that $a$ is  odd. Thus, $2\nmid a-2^{c}$. For
$\ell +4\le i\le s$ with $p_i>2$, we have
$$2\left( a-2^{c}\right)\equiv 2c_i-2^{c+1}\equiv 2^{3\cdot 2^{\ell }}+2-2^{3\cdot 2^{\ell
}} \equiv 2\pmod{p_{i}}.$$ It follows that $p_{i}\nmid a-2^{c}$.

Up to now, we have proved that for all $1\le i\le s$ with $i\not=
\ell+3$, $p_{i}\nmid a-2^{c}$. That is, (iv) holds.

This completes the proof of Theorem \ref{thm1}.

\section{A negative answer to Conjecture A}\label{ConjectureAsec}

It is clear that Theorem \ref{thm1} implies that Conjecture A is
false. In this section, we give a different proof for this.

\begin{theorem}\label{thm1a} Conjecture A is false.
\end{theorem}

Suppose that Conjecture A is true. Then $\mathcal{U}$ can be
written as
$$\mathcal{U}=\{ m_0h+a_0 : h=0,1,\dots \} \cup W,$$
where $m_0, a_0$ are two positive integers and $W$ has asymptotic
density zero. It is clear that $m_0$ is even. Otherwise,
$\mathcal{U}$ contains infinitely many even integers.

We begin with the following simple lemmas.

\begin{lemma}\label{basiclemma1} Let $a, m$ be two positive
integers such that \begin{equation}\label{eq1}\{ mh+a :
h=0,1,\dots \} \subseteq \mathcal{U} .\end{equation} Then $m_0\mid
m$ and $m_0\mid a-a_0$. \end{lemma}

\begin{proof} Since the
asymptotic density of $W$ is zero, the set of positive integers
$h$ with
$$\{ mh+a, m(h+1)+a \} \cap W \not=\emptyset $$
has asymptotic density zero.  It follows from \eqref{eq1} that
there is an integer $h_0$ such that
$$\{ mh_0+a, m(h_0+1)+a \} \subseteq \{ m_0h+a_0 : h=0,1,\dots \} $$
So $m=m(h_0+1)+a-(mh_0+a)$ is divisible by $m_0$. Noting that
$mh_0+a\equiv a_0\pmod{m_0}$ and $m_0\mid m$, we have $m_0\mid
a-a_0$.
\end{proof}

\begin{lemma}\label{firstap} We have
\begin{equation}\label{eq4a}\{ 11184810s + 992077 : s=0,1, \dots \} \subseteq \mathcal{U} .\end{equation}
\end{lemma}

\begin{proof} We follow the proof of Erd\H os \cite{Erdos1950}.
It is easy to see that
\begin{eqnarray*}\mathbb{Z}&=&0\hskip -3mm\pmod 2\cup 1\hskip -3mm\pmod 4\cup 3\hskip -3mm\pmod 8 \cup 1\hskip -3mm\pmod{3}\\
&& \cup\, 3\hskip -3mm\pmod{12}\cup 23\hskip
-3mm\pmod{24}.\end{eqnarray*} Thus, for any positive integer $k$,
\begin{eqnarray*}2^k&\in & 2^0\hskip -3mm\pmod{(2^2-1)}\cup 2^1\hskip -3mm\pmod{(2^4-1)}\cup 2^3\hskip -3mm\pmod {(2^8-1)}\\
&&\cup\,  2^1\hskip -3mm\pmod{(2^3-1)} \cup 2^3\hskip
-3mm\pmod{(2^{12}-1)}\cup 2^{23}\hskip
-3mm\pmod{(2^{24}-1)}.\end{eqnarray*} Noting that
$$3\mid 2^2-1,\ 5\mid 2^4-1,\ 17\mid 2^8-1,\ 7\mid 2^3-1,\ 13\mid
2^{12}-1,\ 241\mid 2^{24}-1,$$  for any positive integer $k$, we
have
\begin{eqnarray}\label{eq4}2^k&\in & 2^0\hskip -3mm\pmod{3}\cup 2^1\hskip -3mm\pmod{5}
\cup 2^3\hskip -3mm\pmod {17}\cup 2^1\hskip -3mm\pmod{7}\nonumber\\
&& \cup\,  2^3\hskip -3mm\pmod{13}\cup 2^{23}\hskip
-3mm\pmod{241}.\end{eqnarray}  A calculation shows that
\begin{eqnarray}\label{eq5}&&1\hskip -3mm\pmod 2\cap 2^0\hskip -3mm\pmod{3}\cap 2^1\hskip -3mm\pmod{5}\cap 2^3\hskip -3mm\pmod {17}\cap
2^1\hskip -3mm\pmod{7}\nonumber\\
&& \cap\,  2^3\hskip -3mm\pmod{13}\cap 2^{23}\hskip -3mm\pmod{241}\nonumber\\
&=& \{ 11184810s + 992077 : s\in \mathbb{Z} \} .\end{eqnarray} Let
$$n\in \{ 11184810s + 992077 : s=0,1, \dots \} .$$
Suppose that  $n\notin \mathcal{U}$. Then $n$ can be represented
as $p+2^k$, $p\in \mathcal{P}$, $k\in \mathbb{N}$. By \eqref{eq4}
and \eqref{eq5}, we have
\begin{eqnarray*}p=n-2^k&\in & 0\hskip -3mm\pmod{3}\cup 0\hskip -3mm\pmod{5}\cup 0\hskip -3mm\pmod {17}\\
&& \cup\,  0\hskip -3mm\pmod{7}\cup 0\hskip -3mm\pmod{13}\cup
0\hskip -3mm\pmod{241}.\end{eqnarray*} It follows that
$$p\in \{ 3,5,17,7,13,241\} .$$
By \eqref{eq5}, we have
$$p=n-2^k\equiv 1-2^k\not\equiv 1 \hskip -3mm\pmod{3},$$
$$p=n-2^k\equiv 2-2^k\equiv 0, 1, 5  \hskip -3mm\pmod{7},$$
$$p=n-2^k\equiv 8-2^k\equiv 0,4,6,7,9,10,12,16 \hskip -3mm\pmod{17},$$
The first congruence gives $p\not= 7, 13, 241$. The second
congruence gives $p\not= 3, 17$. The third congruence gives
$p\not= 5$. Thus, we have derived a contradiction. Hence, $n\in
\mathcal{U}$. Therefore, \eqref{eq4a} holds.
\end{proof}

\begin{lemma}\label{secondap} We have
\begin{equation}\label{eq4b}\{ 11184810s + 3292241 : s=0,1, \dots \} \subseteq \mathcal{U} .\end{equation}
\end{lemma}

\begin{proof} We follow the proof of Lemma \ref{firstap}. It is
easy to see that
\begin{eqnarray*}\mathbb{Z}&=&1\hskip -3mm\pmod 2\cup 0\hskip -3mm\pmod 4\cup 2\hskip -3mm\pmod 8\cup  0\hskip -3mm\pmod{3} \\
&& \cup\, 2\hskip -3mm\pmod{12}\cup 22\hskip
-3mm\pmod{24}.\end{eqnarray*} As in the proof of Lemma
\ref{firstap}, for any positive integer $k$,
\begin{eqnarray}\label{eq4w}2^k&\in & 2^1\hskip -3mm\pmod{3}\cup 2^0\hskip -3mm\pmod{5}\cup 2^2\hskip -3mm\pmod {17}\cup 2^0\hskip -3mm\pmod{7}\nonumber\\
&& \cup\,  2^2\hskip -3mm\pmod{13}\cup 2^{22}\hskip
-3mm\pmod{241}.\end{eqnarray}  A calculation shows that
\begin{eqnarray}\label{eq5w}&&1\hskip -3mm\pmod 2\cap 2^1\hskip -3mm\pmod{3}\cap 2^0\hskip -3mm\pmod{5}\cap 2^2\hskip -3mm\pmod {17}\cap
2^0\hskip -3mm\pmod{7}\nonumber\\
&&\cap\,  2^2\hskip -3mm\pmod{13}\cap 2^{22}\hskip -3mm\pmod{241}\nonumber\\
&=& \{ 11184810s + 3292241 : s\in \mathbb{Z} \} .\end{eqnarray}
Let
$$n\in \{ 11184810s + 3292241 : s=0,1, \dots \} .$$
Suppose that  $n\notin \mathcal{U}$. Then  $n$ can be represented
as $p+2^k$, $p\in \mathcal{P}$, $k\in \mathbb{N}$. By \eqref{eq4w}
and \eqref{eq5w}, we have
\begin{eqnarray*}p=n-2^k&\in & 0\hskip -3mm\pmod{3}\cup 0\hskip -3mm\pmod{5}\cup 0\hskip -3mm\pmod {17}\cup 0\hskip -3mm\pmod{7}\\
&&\cup\,  0\hskip -3mm\pmod{13}\cup 0\hskip
-3mm\pmod{241}.\end{eqnarray*} It follows that
$$p\in \{ 3,5,17,7,13,241\} .$$
By \eqref{eq5w}, we have
$$p=n-2^k\equiv 1-2^k\equiv 0, 4, 6  \hskip -3mm\pmod{7},$$
$$p=n-2^k\equiv 4-2^k\equiv 0,2,3,5,6,8,12, 13\hskip -3mm\pmod{17}.$$
The first congruence gives $p\not= 3, 5, 17, 241$. The second
congruence gives $p\not= 7$. The remaining case is $p=13$. Thus,
by \eqref{eq5w},
$$2^k=n-p\equiv 2-13\equiv 1\hskip -3mm\pmod{3},$$ $$2^k=n-p\equiv 1-13\equiv 3\hskip -3mm\pmod{5}.$$ It
follows that $k\equiv 0\pmod{2}$ and $k\equiv 3\pmod{4}$, a
contradiction. Hence,  $n\in \mathcal{U}$.  Therefore,
\eqref{eq4b} holds.
\end{proof}

\begin{proof}[Proof of Theorem \ref{thm1a}] By Lemmas
\ref{basiclemma1}$-$\ref{secondap}, we have
$$m_0\mid 11184810,\quad m_0\mid 992077-a_0, \quad m_0\mid
3292241-a_0.$$ It follows that
$$m_0\mid 11184810,\quad m_0\mid 3292241-992077.$$
Noting that $$\gcd ( 11184810, 3292241-992077)=2,$$ we have
$m_0\mid 2$. So $m_0=2$. Hence $$\mathcal{U}=\{ m_0k+a_0 :
k=0,1,\dots \} \cup W$$ contains all sufficiently large odd
integers. But,  $2^n+3\notin \mathcal{U}$ for all $n\ge 1$, a
contradiction.
\end{proof}

\begin{remark}It is clear that \eqref{eq1} can be replaced by \begin{equation*}\{ mh+a :
h=0,1,\dots \} \subseteq \mathcal{U}\cup T,\end{equation*} where
$T$ is a set of asymptotic density  zero. Thus, Lemmas
\ref{firstap} and \ref{secondap} in the following weak forms can
be applied to prove Theorem \ref{thm1a}: $$\{ 11184810s + 992077 :
s=0,1, \dots \} \subseteq \mathcal{U}\cup T,$$ $$\{ 11184810s +
3292241 : s=0,1, \dots \} \subseteq \mathcal{U} \cup T,$$ where
$$T=\{ 2^k +p : k\in \mathbb{N}, p\in \{ 3,5,17,7,13,241\} \} .$$
\end{remark}

\section{Proofs of Theorem \ref{thm1b}, Theorem \ref{thm1c} and their corollaries} \label{thm1bsec}

 A sequence $\{ m_1,
\dots , m_t\}$ of positive integers is called \emph{coverable} if
there exist $t$ integers $a_1, \dots , a_t$ such that
$$\bigcup_{i=1}^t a_i\pmod{m_i} =\mathbb{Z}.$$
A coverable sequence $\{ m_1, \dots , m_t\}$  is called
\emph{minimal} if no proper subsequence of $\{ m_1, \dots , m_t
\}$ is coverable. For example, $\{ 3, 3, 3\}$ is a  minimal
coverable sequence. It is clear that if $\{ m_1, \dots , m_t \} $
is coverable sequence, then
\begin{equation}\label{a3}\sum_{i=1}^t\frac 1{m_i}\ge
1.\end{equation}  A set $\{ p_1, \dots , p_t\} $ of distinct odd
primes is called a \emph{well constructed  prime set} if $\{
r_{p_1}, \dots , r_{p_t}\} $ is coverable. A well constructed
prime set $\{ p_1, \dots , p_t\} $ is called \emph{minimal} if no
proper subset of $\{ p_1, \dots , p_t\} $ is a well constructed
prime set.

\begin{lemma}\label{lemcover1} Let  $\{ m_1, \dots , m_t\} $  be a
sequence and let $a_1,\dots , a_t$ be integers such that
\begin{equation}\label{a2}\bigcup_{i=1}^t a_i\pmod{m_i} =\mathbb{Z}.\end{equation}
If $p$ is a prime and $\{ a_i : p\mid m_i\} $ does not contain a
complete residue system modulo $p$, then $\{ m_i : 1\le i\le t,
p\nmid m_i \} $ is  a  coverable sequence.

In particular, if $\{ m_1, \dots , m_t\} $ is  a  coverable
sequence and $|\{ i : p\mid m_i \} |<p$, then $\{ m_i : p\nmid m_i
\} $ is also a  coverable sequence.
\end{lemma}

\begin{proof}
Without loss of generality, we may assume that $p\mid m_i$ $(1\le
i\le s)$ and $p\nmid m_j$ $(s<j\le t)$. Then for $s<j\le t$, there
exists an integer $u_j$ such that $pu_j\equiv 1\pmod{m_j}$. If $\{
a_i : p\mid m_i\} $ does not contain a complete residue system
modulo $p$ (including the case $\{ a_i : p\mid m_i\} =\emptyset$),
then there exists an integer $a$ such that $a\not\equiv
a_i\pmod{p}$ $(1\le i\le s)$. Thus,
$$a\pmod p \cap a_i\pmod{m_i} =\emptyset, \quad 1\le i\le s.$$
It follows from \eqref{a2} that
$$a\pmod p \subseteq \bigcup_{j=s+1}^t a_j\pmod{m_j}.$$
That is, for any integer $n$, there exists $s+1\le j\le t$ such
that
$$a+pn\equiv  a_j\pmod{m_j}.$$
So $$n\equiv (a_j-a)u_j\pmod{m_j}.$$ Hence
$$\bigcup_{j=s+1}^t (a_j-a)u_j\pmod{m_j} =\mathbb{Z}.$$
It follows that $\{ m_{s+1},\dots , m_t\}$ is a coverable
sequence.

If $|\{ i : p\mid m_i \} |<p$, then $\{ a_i : p\mid m_i\} $ does
not contain a complete residue system modulo $p$ and so $\{ m_i :
p\nmid m_i \} $ is a coverable sequence.
\end{proof}

\begin{lemma}\label{lemprimitive} (i) There is no prime $p$ such
that $r_p=1$;

(ii) There is no prime $p$ such that $r_p=6$;

(iii) Let $p$ be a prime. Then $r_p=2\Leftrightarrow p=3$,
$r_p=3\Leftrightarrow p=7$, $r_p=2^2\Leftrightarrow p=5$,
$r_p=5\Leftrightarrow p=31$, $r_p=2^3\Leftrightarrow p=17$,
$r_p=9\Leftrightarrow p=73$, $r_p=10\Leftrightarrow p=11$,
$r_p=12\Leftrightarrow p=13$, $r_p=15\Leftrightarrow p=151$,
$r_p=2^4\Leftrightarrow p=257$, $r_p=2^5\Leftrightarrow
p=65537$.\end{lemma}

Lemma \ref{lemprimitive} can be verified directly. We omit its
proof.

\begin{lemma}\label{lemcover2} If $\{ p_1, \dots , p_t\} $ is
 a well constructed prime set, then $t\ge 6$ and $p_1\cdots p_t\ge 5592405.$

 Furthermore, $t=6$ if and only if $p_1\cdots p_t=5592405.$
 \end{lemma}

\begin{proof} Without loss of generality, we may assume that $\{ p_1, \dots , p_t\} $ is
 a minimal well constructed prime set. By the definition, $p_1, \dots ,
 p_t$ are distinct odd primes and $\{ r_{p_1}, \dots ,
r_{p_t}\}$ is  a minimal coverable sequence. By \eqref{a3},
\begin{equation}\label{a4}\sum_{i=1}^t\frac 1{r_{p_i}}\ge
1.\end{equation} If there exists a prime $p\ge 5$ such that $p\mid
r_{p_1} \cdots r_{p_t}$, then by Lemma \ref{lemcover1}, there are
at least $p$ of $r_{p_1}, \dots , r_{p_t}$ with  $p\mid r_{p_i}$.
It follows that $t\ge p$.  Since $r_{p_i}$ is the multiplicative
order of $2$ modulo $p_i$, we have $r_{p_i}\mid p_i-1$. So $p\mid
p_i-1$ for at least $p$ of $i$. Hence
$$p_1\cdots p_t\ge (2p+1) (4p+1)(6p+1)(8p+1)(10p+1)> 5592405.$$
If $p\ge 7$, then $t\ge p\ge 7$. Now we assume that $p=5$ and
$t\le 6$. If $5\mid r_{p_i}$ for all $1\le i\le t$, then by Lemma
\ref{lemprimitive},
$$\sum_{i=1}^t\frac 1{r_{p_i}}\le \frac 15 +\frac{5}{10}<1,$$
a contradiction with \eqref{a4}. So there exists $1\le j\le t$
with $5\nmid r_{p_j}$. Without loss of generality, we may assume
that $5\nmid r_{p_1}$. Noting that $t\le 6$ and there are at least
$5$ of $p_i $ $(1\le i\le t)$ with $5\mid r_{p_i}$,  we have $t=6$
and $5\mid r_{p_i}$ $(2\le i\le 6)$. By Lemma \ref{lemprimitive},
$$\sum_{i=1}^t\frac 1{r_{p_i}}\le \frac 12+\frac 15+\frac 1{10}+\frac 1{15} +\frac{2}{20}<1,$$
a contradiction with \eqref{a4}. Up to now, we have proved that if
there exists a prime $p\ge 5$ such that $p\mid r_{p_1} \cdots
r_{p_t}$, then  $t\ge 7$ and $p_1\cdots p_t>5592405.$

Now we assume that there is no prime $p\ge 5$ such that $p\mid
r_{p_1} \cdots  r_{p_t}$. So each $r_{p_i}$ has only prime
divisors $2$ and $3$. The initial primes $p$ for which $r_{p}$ has
only prime divisors $2$ and $3$ are $3, 5,
7,13,17,19,37,73,97,109,163,193,241$. If $t\ge 7$, then
$$p_1\cdots p_t\ge 3\cdot 5\cdot 7\cdot 13\cdot 17\cdot 19\cdot 37>
5592405.$$ In the following, we assume that $t\le 6$. The initial
positive integers of the form $2^a3^b$ are
$$1,2,3,4,6,8,9,12,16,18,24,27,32,36,\dots .$$  If $r_{p_i}\not= 2$,
then by Lemma \ref{lemprimitive},
$$\sum_{i=1}^t\frac 1{r_{p_i}}\le \frac 13+\frac 14+\frac 18+\frac 19+\frac 1{12}+\frac
1{16}<1,$$ a contradiction with \eqref{a4}. So there exists $1\le
i\le t$ such that $r_{p_i}=2$. Without loss of generality, we may
assume that $r_{p_1}=2$.

We divide into the following cases:

{\bf Case 1:} All $r_{p_i}$ are powers of $2$.  By Lemma
\ref{lemprimitive},
$$\sum_{i=1}^t\frac 1{r_{p_i}}\le \frac 12+\frac 1{2^2}+\frac 1{2^3}
+\frac 1{2^4}+\frac 1{2^5}+\frac 1{2^6}<1,$$ a contradiction with \eqref{a4}.

{\bf Case 2:} At least one of $r_{p_i}$ is divisible by $3$. Since
$\{ r_{p_1}, \dots , r_{p_t}\}$ is  a coverable sequence, there
exist $t$ integers $a_1, \dots , a_t$ such that
\begin{equation}\label{a5}\bigcup_{i=1}^t a_i\pmod{r_{p_i}} =\mathbb{Z}.\end{equation} Let
$$A=\{ i : 3\nmid r_{p_i} \} , \quad B_j=\{ i : 3\mid r_{p_i}, a_i\equiv
j\pmod 3\} ,\quad 0\le j\le 2.$$ Then $1\in A$ by $r_{p_1}=2$,
$|A|<t$ and $|A|+|B_0|+|B_1|+|B_2|=t\le 6$. Since $\{ r_{p_1},
\dots , r_{p_t}\}$ is a minimal coverable sequence and $|A|<t$, it
follows that $\{ r_{p_i} : i\in A\}$ is not a coverable sequence.
By Lemma \ref{lemcover1}, $\{ a_i : 3\mid r_{p_i} \}$ contains a
residue system modulo $3$. So $B_j\not=\emptyset $ $(j=0,1,2)$.
That is, $|B_j|\ge 1$ $(j=0,1,2)$. Noting that
$|A|+|B_0|+|B_1|+|B_2|\le 6$, we have $1\le |A|\le 3$. We divide
into the following subcases
according to $|A|=1,2,3$.\\

{\bf Subcase 2.1:} $|A|=1$. Since $|A|+|B_0|+|B_1|+|B_2|=t\le 6$
and $|B_j|\ge 1$ $(j=0,1,2)$, we know that at least one of $|B_j|$
$(j=0,1,2)$ is $1$. Suppose that $|B_{j_0}|=1$. Without loss of
generality, we may assume that $B_{j_0}=\{ 2 \}$. For any integer
$n$, by \eqref{a5} we have
$$3n+j_0\in \bigcup_{i=1}^t a_i\pmod{r_{p_i}}.$$
Noting that $\{ 3, \dots , t\} =(B_0\cup B_1\cup B_2)\setminus
B_{j_0}$, we have
$$3n+j_0\notin \bigcup_{i=3}^t a_i\pmod{r_{p_i}}.$$
It follows that
$$3n+j_0\in a_1\pmod{r_{p_1}}\cup a_2\pmod{r_{p_2}}.$$
In view of $r_{p_1}=2$, $a_2\equiv j_0\pmod 3$ and $3\mid
r_{p_2}$, we have
$$n\in (a_1-j_0)\pmod 2 \cup 3^{-1}(a_2-j_0)\pmod{(r_{p_2}/3)}.$$
Hence, $\{ 2, r_{p_2}/3 \}$ is a coverable sequence. It follows
from \eqref{a3} and Lemma \ref{lemprimitive} that $r_{p_2}/3=1$.
That is, $r_{p_2}=3$ and $p_2=7$. For $3\le i\le t$, by Lemma
\ref{lemprimitive} and $p_i\not= p_2$, we have $r_{p_i}/3\ge 3$.
Noting that $p_1,\dots , p_t$ are distinct, $j_0$ is the unique
$j\in \{ 0,1,2\}$ such that $|B_{j}|=1$. Thus, $|B_j|\not=1$ for
$j\not= j_0$.  By $|A|+|B_0|+|B_1|+|B_2|=t\le 6$ and $|B_j|\ge 1$
$(j=0,1,2)$, we have $t=6$ and $|B_j|=2$ for $j\not= j_0$. Without
loss of generality, we may assume that $B_{j_1}=\{ 3,4 \}$ and
$B_{j_2}=\{ 5,6 \}$. Similar to the above arguments, both $\{ 2,
r_{p_3}/3, r_{p_4}/3 \}$ and $\{ 2, r_{p_5}/3, r_{p_6}/3 \}$ are
coverable sequences. Since $r_{p_i}/3\ge 3$ $(3\le i\le 6)$, it
follows from \eqref{a3} that both $\{ 2, r_{p_3}/3, r_{p_4}/3 \}$
and $\{ 2, r_{p_5}/3, r_{p_6}/3 \}$ are minimal coverable
sequences. By Lemma \ref{lemcover1}, $3\nmid (r_{p_i}/3)$, $(3\le
i\le 6)$. In view of \eqref{a3}, the only possibility is
$r_{p_i}/3=4$ $(3\le i\le 6)$. That is, $r_{p_i}=12$ $(3\le i\le
6)$.  By Lemma \ref{lemprimitive}, $p_i=13$ $(3\le i\le 6)$. This
contradicts that $p_1, \dots , p_t$ are distinct.\\

{\bf Subcase 2.2:} $|A|=2$. Since $|A|+|B_0|+|B_1|+|B_2|=t\le 6$
and $|B_j|\ge 1$ $(j=0,1,2)$, we know that at least two of $|B_j|$
$(j=0,1,2)$ are $1$. Without loss of generality, we may assume
that
$$ A=\{ 1 ,2 \} ,\quad B_{j_1}=\{ 3 \},\quad B_{j_2}=\{ 4 \} .$$
By $r_{p_1}=2$ and Lemma \ref{lemprimitive}, we have $r_{p_2}\ge
4$. Similar to the arguments in Subcase 2.1, both $\{ r_{p_1},
r_{p_2}, r_{p_3}/3 \} $ and $\{ r_{p_1}, r_{p_2}, r_{p_4}/3 \} $
are  coverable sequences. In view of Lemma \ref{lemprimitive}, one
of $r_{p_3}$ and $r_{p_4}$ is not equal to $3$. Without loss of
generality, we may assume that $r_{p_3}\not=3$. By Lemma
\ref{lemprimitive},  $r_{p_3}/3\ge 3$. From \eqref{a3}, we obtain
that $\{ r_{p_1}, r_{p_2}, r_{p_3}/3 \} $ is a minimal coverable
sequence. In view of Lemma \ref{lemcover1}, $3\nmid r_{p_3}/3$.
Thus, $r_{p_1}, r_{p_2}, r_{p_3}/3$ are all powers of $2$. Again,
by \eqref{a3}, $\{ r_{p_1}, r_{p_2}, r_{p_3}/3 \} =\{ 2, 4, 4\}$.
It follows from Lemma \ref{lemprimitive} that $\{ p_1, p_2, p_3\}
=\{ 3, 5, 13\} $. If $r_{p_4}\not= 3$, then the similar arguments
give $\{ p_1, p_2, p_4\} =\{ 3, 5, 13\} $, a contradiction. So
$r_{p_4}= 3$ and then $p_4=7$. Basing on the above arguments and
$p_1,\dots , p_t$ being distinct, we have $|B_{j_3}|\not= 1$ for
$j_3\in \{ 0,1,2\} \setminus \{ j_1, j_2\}$. So $t=6$ and
$|B_{j_3}|=2$ by $|A|+|B_0|+|B_1|+|B_2|=t\le 6$. Hence,
$B_{j_3}=\{ 5,6 \}$. Similar to the arguments in Subcase 2.1,  $\{
r_{p_1}, r_{p_2}, r_{p_5}/3, r_{p_6}/3 \} $ is a coverable
sequence. That is, $\{ 2, 4, r_{p_5}/3, r_{p_6}/3 \} $ is a
coverable sequence. By $p_i\notin \{ p_3, p_4 \} =\{ 7,13\} $
$(i=5,6)$ and Lemma \ref{lemprimitive}, we have $r_{p_i}\not= 3,
6, 12$ $(i=5,6)$. If $3\mid r_{p_5}/3$ and $3\mid r_{p_6}/3$, then
by Lemma \ref{lemcover1}, $\{ 2, 4\}$ is a coverable sequence, a
contradiction with \eqref{a3}. If $3\nmid r_{p_5}/3$ and  $3\mid
r_{p_6}/3$, then  by Lemma \ref{lemcover1}, $\{ 2, 4, r_{p_5}/3\}$
is a coverable sequence. By \eqref{a3}, we have $r_{p_5}/3=1,2,4$,
a contradiction with $r_{p_5}\not= 3, 6, 12$. Similarly, we can
derive a contradiction if $3\mid r_{p_5}/3$ and $3\nmid
r_{p_6}/3$. Hence,  $3\nmid r_{p_5}/3$ and  $3\nmid r_{p_6}/3$. By
$r_{p_i}\not= 3, 6, 12$  $(i=5,6)$, we have $r_{p_5}/3\ge 8$ and $
r_{p_6}/3\ge 8$. Since $\{ 2, 4, r_{p_5}/3, r_{p_6}/3 \} $ is a
coverable sequence, it follows from \eqref{a3} that
$r_{p_5}/3=r_{p_6}/3=8$. That is, $r_{p_5}=r_{p_6}=24$. By  Lemma
\ref{lemprimitive},
$p_5=p_6=241$, a contradiction.\\

{\bf Subcase 2.3:} $|A|=3$. Since $|A|+|B_0|+|B_1|+|B_2|=t\le 6$
and $|B_j|\ge 1$ $(j=0,1,2)$, it follows that $t=6$ and $|B_j|=1$
$(j=0,1,2)$. Without loss of generality, we may assume that
$$ A=\{ 1 ,2, 3 \} ,\quad B_{j}=\{ 4+j \}, \quad j=0,1,2.$$
Similar to the  arguments in Subcase 2.1, for each $4\le j\le 6$,
$\{ r_{p_1}, r_{p_2}, r_{p_3}, r_{p_j}/3 \} $ is a coverable
sequence. It follows from \eqref{a3} that
\begin{equation}\label{a8} \frac 1{r_{p_1}}+\frac 1{r_{p_2}}+\frac 1{r_{p_3}}+\frac
1{r_{p_j}/3}\ge 1,\quad 4\le j\le 6.\end{equation} Given $4\le
j\le 6$. If $3\mid r_{p_j}/3$, then by Lemma \ref{lemcover1}, $\{
r_{p_1}, r_{p_2}, r_{p_3}\} $ is a coverable sequence, a
contradiction with $\{ r_{p_1}, \dots , r_{p_t}\} $ being a
minimal coverable sequence.  Hence, $3\nmid r_{p_j}/3$. By  Lemma
\ref{lemprimitive}, there exists $4\le j\le 6$ with $r_{p_j}/3\ge
2^3$. Without loss of generality, we may assume that $r_{p_4}/3\ge
2^3$. Again, by  Lemma \ref{lemprimitive},
\begin{equation}\label{a9}\frac 1{r_{p_1}}+\frac 1{r_{p_2}}
+\frac 1{r_{p_3}}+\frac 1{r_{p_4}/3}\le \frac 12+\frac 14+\frac
18+\frac 18=1.\end{equation} In view of \eqref{a8} and \eqref{a9},
we have
$$ \{ r_{p_1}, r_{p_2}, r_{p_3}, r_{p_4}/3 \} =\{ 2, 4, 8, 8 \}
.$$ It follows from Lemma \ref{lemprimitive} that $$\{ p_1, p_2,
p_3, p_4\} =\{ 3, 5, 17, 241 \} .$$
 If $r_{p_5}/3\ge 2^3$, then by the
similar arguments, we have $$\{ p_1, p_2, p_3, p_5\} =\{ 3, 5, 17,
241 \} ,$$ a contradiction with $p_4\not= p_5$. So $r_{p_5}/3<
2^3$. By  Lemma \ref{lemprimitive}, we have $r_{p_5}/3\in \{ 1, 4
\}$. Similarly, $r_{p_6}/3\in \{ 1, 4 \}$. It follows from Lemma
\ref{lemprimitive} that $\{ p_5, p_6 \} =\{ 7, 13\} $. Therefore,
$$p_1\cdots p_t= 3\cdot 5\cdot 7\cdot 13\cdot 17\cdot 241=5592405.$$
This completes the proof of Lemma \ref{lemcover2}.\end{proof}

\begin{lemma}\cite{XGSun2010}\label{lemcover3} Let $a$ and $m$ be
 two positive integers. If there exists an integer $\ell \ge 1$
such that $(a-2^\ell , m)=1$, then there is a positive proportion
of positive odd integers in $\{ mh+a : h=0,1,\dots \}$ that can be
written as $p+2^k$, where $p\in \mathcal{P}$ and $k\in
\mathbb{N}$.
 \end{lemma}

\begin{proof}[Proof of Theorem \ref{thm1b}] Let $m,a$ be two positive integers
such that \begin{equation}\label{t1}\{ mh+a : h=0,1,\dots \}
\setminus \mathcal{U}\end{equation} has asymptotic density zero.
Then $2\mid m$ and $2\nmid a$, otherwise, the even integers in $\{
mh+a : h=0,1,\dots \}$ has asymptotic density $(2m)^{-1}$ or
$m^{-1}$, a contradiction with \eqref{t1}. Let $p_1,\dots ,p_t$ be
all distinct odd prime divisors $p$ of $m$ for which there exists
an integer $\ell \ge 1$ such that $a\equiv 2^\ell \pmod{p}$.  Let
$a_i\ge 1$ $(1\le i\le t)$ be integers with $a\equiv 2^{a_i}
\pmod{p_i}$. By Lemma \ref{lemcover3}, $(a-2^\ell , m)>1$ for all
integers $\ell \ge 1$. This implies that for each positive integer
$\ell$, there exists $1\le i\le t$ such that $p_i\mid a-2^\ell$,
i.e. $p_i\mid 2^{a_i}-2^\ell$, and so $\ell \equiv
a_i\pmod{r_{p_i}}$. It follows that for each integer $\ell$, we
have $\ell \equiv a_i\pmod{r_{p_i}}$ for some $1\le i\le t$. This
means that $\{ p_1,\dots ,p_t\} $ is a well constructed prime set.
By Lemma \ref{lemcover2}, $t\ge 6$ and $p_1\cdots p_t\ge 5592405$,
and $t=6$ if and only if $p_1\cdots p_t=5592405$. It follows that
$\omega (m)=1+t\ge 7$ and $m\ge 2p_1\cdots p_t\ge 11184810$ , and
$\omega (m)=7$ if and only if $m= 11184810$. By Lemma
\ref{firstap}, $$\{ 11184810s + 992077 : s=0,1, \dots \} \subseteq
\mathcal{U} .$$ It follows that $\min m=11184810$ and $\min \omega
(m)=7$.
\end{proof}

\begin{proof}[Proof of Corollary \ref{cor2}] In the proof of Theorem
\ref{thm1b}, replacing \eqref{t1} by
$$\{ mh+a : h=0,1,\dots \} \subseteq \mathcal{U},$$
we can obtain a proof of Corollary \ref{cor2}.
\end{proof}

\begin{proof}[Proof of Theorem \ref{thm1c}] Let $\{ 11184810h+b : h=0,1,\dots \}$ be  a
longest quasi-non-representable infinite arithmetic progression.
Then $2\nmid b$. By Lemma \ref{lemcover3}, we have $(b-2^k,
11184810)>1$ for all positive integers $k$.  Let $a$ be the
integer with $0\le a<11184810$ such that $b\equiv
a\pmod{11184810}$. Then $$(a-2^k, 11184810)=(b-2^k, 11184810)>1$$
for all  integers $1\le k\le 24$. Since $2\nmid b$, we have
$2\nmid a$. A calculation shows that odd integers $a$ with $0\le
a<11184810$ and $(a-2^k, 11184810)>1$ for all integers $1\le k\le
24$ are in the list \eqref{list}.  Since
$$ \{ 11184810h+b : h=0,1,\dots \} \subseteq \{ 11184810h+a : h=0,1,\dots
\} ,$$ it is enough to prove that $\{ 11184810h+a : h=0,1,\dots
\}$ is  a longest quasi-non-representable infinite arithmetic
progression for all $a$ in the list \eqref{list}.

Let $a$ be an integer in the list \eqref{list}. One can verify
that $(a-2^k, 11184810)>1$ for all integers $1\le k\le 24$. Since
$a$ is odd and $(11184810/2)\mid 2^{24}-1$, it follows that
$(a-2^k, 11184810)>1$ for all positive integers $k$.

If
$$ \{ 11184810h+a : h=0,1,\dots
\}  \subseteq \{ mh+a' : h=0,1,\dots \},$$ where $\{ mh+a' :
h=0,1,\dots \}$ is quasi-non-representable, then $m\le 11184810$
and $a'>0$. On the other hand, by Theorem \ref{thm1b}, $m\ge
11184810$. So $m=11184810$. In view of $a'>0$, $0\le a<11184810$
and
$$ \{ 11184810h+a : h=0,1,\dots
\}  \subseteq \{ 11184810h+a' : h=0,1,\dots \} ,$$ we have $a=a'$.
Hence
$$\{ 11184810h+a : h=0,1,\dots
\}  = \{ mh+a' : h=0,1,\dots \}.$$

The remaining thing is to prove that for each $a$ in the list
\eqref{list},
\begin{equation*} \{ 11184810h+a : h=0,1,\dots
\}  \setminus  \mathcal{U}\end{equation*} has asymptotic density
zero.

Let
$$n\in \{ 11184810h+a : h=0,1,\dots
\}  \setminus  \mathcal{U}.$$ Clearly, $n$ is odd by $2\nmid a$.
Then there exist $p\in \mathcal{P}$ and $k\in \mathbb{N}$ such
that $n=p+2^k$. Thus,
$$(p,  11184810)=(n-2^k, 11184810)=(a-2^k, 11184810)>1.$$ It
follows that $p\in \{ 3, 5, 7, 13, 17, 241 \} $. Therefore,
\begin{eqnarray*} &&\{ 11184810h+a : h=0,1,\dots
\}  \setminus  \mathcal{U}\\
&\subseteq & \{ 2^k+p : k\in \mathbb{N}, p\in \{ 3, 5, 7, 13, 17,
241 \} \} \end{eqnarray*} has asymptotic density zero.
\end{proof}

\begin{proof}[Proof of Corollary \ref{cor2a}] Suppose that $$\{ 11184810h+b : h=0,1,\dots \} \subseteq
\mathcal{U}.$$ By $b\in \mathcal{U}$, we have $b>0$. Let $a$ be
the integer with $0\le a<11184810$ such that $b\equiv
a\pmod{11184810}$. As the proof of Theorem \ref{thm1c}, $a$ is in
the list \eqref{list}.

Conversely, suppose that $b\ge 0$ and $b\equiv a\pmod{11184810}$
for some $a$ in the list \eqref{list}. Noting that $b\ge 0$ and
$0\le a<11184810$, we have
$$ \{ 11184810h+b : h=0,1,\dots \} \subseteq \{ 11184810h+a : h=0,1,\dots
\}.$$ The remaining thing is to prove that
\begin{equation}\label{aa99} \{ 11184810h+a : h=0,1,\dots
\}  \subseteq  \mathcal{U}.\end{equation} Let
$$n\in \{ 11184810h+a : h=0,1,\dots
\} .$$ Clearly, $n$ is odd by $2\nmid a$. Suppose that $n\notin
\mathcal{U}$. Then there exist $p\in \mathcal{P}$ and $k\in
\mathbb{N}$ such that $n=p+2^k$. As the proof of Theorem
\ref{thm1c}, we have  $p\in \{ 3, 5, 7, 13, 17, 241 \} $. Let
$k'\equiv k\pmod{24}$, $1\le k'\le 24$. Noting that
 $$2^{24}\equiv 1\pmod{11184810/2},$$
we have
$$2^{k'}\equiv 2^k\pmod{11184810/2}.$$
It follows that
$$2^{k'}\equiv 2^k\pmod{11184810}.$$
Thus,
$$a\equiv p+2^k\equiv p+2^{k'}\pmod{11184810}.$$
A calculation shows that for $a$ in the list \eqref{list}, $1\le
k'\le 24$ and $p\in \{ 3, 5, 7, 13, 17, 241 \} $,
$$a\not\equiv p+2^{k'}\pmod{11184810},$$
a contradiction. Hence, $n\in \mathcal{U}$. Therefore,
\eqref{aa99} holds and so
$$\{ 11184810h+b : h=0,1,\dots \} \subseteq
\mathcal{U}.$$ This completes the proof of Corollary \ref{cor2a}.
\end{proof}

\begin{proof}[A new proof of Theorem \ref{thm1a}]  Suppose that Conjecture
A is true. Then $\mathcal{U}$ can be written as
$$\mathcal{U}=\{ m_0h+a_0 : h=0,1,\dots \} \cup W,$$
where $m_0, a_0$ are two positive integers and $W$ has asymptotic
density zero. By Theorem \ref{thm1b}, $m_0\ge 11184810$. It
follows that the asymptotic density $m_0^{-1}$ of $\mathcal{U}$ is
less than or equal to $11184810^{-1}$. In view of Theorem
\ref{thm1c}, the asymptotic density of $\mathcal{U}$ is more than
$11184810^{-1}$, a contradiction.

This completes the proof of Theorem \ref{thm1a}.\end{proof}

\section{ Proofs of Theorems \ref{thm2} and \ref{thm3}}  \label{thm2sec}

\begin{proof}[Proof of Theorem \ref{thm2}]
Suppose that $$a\in \bigcup_{i\in I} A_i.$$ Say
$$a\in \{ mh+b: h=0,1,\dots \} \subseteq \mathcal{U}.$$
By Lemma \ref{lemcover3}, for any positive integer $k$, we have
$(b-2^k, m)>1$. It follows that for any positive integer $k$, we
have $(a-2^k, m)=(b-2^k, m)>1$.

Conversely, suppose that there is an integer $m>1$ such that for
any positive integer $k$, $(a-2^k, m)>1$. We choose an integer
$k_0$ with $2^{k_0-1}> \max\{ a, m\} $. Let
$$n\in \{ 2^{k_0} mh+a : h=0,1,\dots \} .$$
Suppose that $n\notin \mathcal{U}$.  Then $n$ can be represented
as $p+2^k$, $p\in \mathcal{P}$, $k\in \mathbb{N}$. By $2\nmid a$,
we have $p\ge 3$. Since $(p, m)=(n-2^k, m)=(a-2^k, m)>1$, it
follows that $p\mid m$ and $p\mid a-2^k$. If $k\ge k_0$, then
$$a\equiv n\equiv p+2^k\equiv p\pmod{2^{k_0}}.$$
It follows from $2^{k_0-1}> \max\{ a, m\} $ and $p\mid m$ that
$a=p$, a contradiction with $p\mid a-2^k$ and $p\ge 3$.
 Hence, $k\le k_0-1$. Since $$a\equiv n\equiv p+2^k\pmod{2^{k_0}}$$
and $$1\le a\le \max\{ a, m\}<2^{k_0}, \quad 1<p+2^k\le
m+2^{k_0-1}<2^{k_0},$$ it follows that $a=p+2^k$, a contradiction
with $a\in \mathcal{U}$. So $n\in \mathcal{U}$. That is,
$$\{ 2^{k_0} mh+a : h=0,1,\dots \} \subseteq \mathcal{U}.$$
It follows that
$$a\in \bigcup_{i\in I} A_i.$$

This completes the proof of Theorem \ref{thm2}.\end{proof}

\begin{proof}[Proof of Theorem \ref{thm3}]

Let $$a\in \bigcup_{i\in I} A_i.$$ Then $a$ is odd by
$A_i\subseteq \mathcal{U}$ $(i\in I)$. By Theorem \ref{thm2},
there is an integer $m>1$ such that for each positive integer $k$,
$(a-2^k, m)>1$. Without loss of generality, by $2\nmid a$ we may
assume that $m$ is a squarefree odd integer. We choose an integer
$k_0$ with $2^{k_0-1}> \max\{ a, m\} $. As in the proof of Theorem
\ref{thm2}, we have
$$\{ 2^{k_0} mh+a : h=0,1,\dots \} \subseteq \mathcal{U}.$$
This implies that
$$a\in \bigcup_{i\in J} A_i.$$
Hence
$$\bigcup_{i\in I} A_i\subseteq \bigcup_{i\in J}
A_i.$$ Since $$\bigcup_{i\in J} A_i\subseteq \bigcup_{i\in I}
A_i,$$ it follows that
$$\bigcup_{i\in I} A_i=\bigcup_{i\in J} A_i.$$

Assume that  there are infinitely many Mersenne primes. We choose
a Mersenne prime $q=2^r-1$ with $2^{r-2}> \max\{ a, m\} $. Let
$$n\in \{ mqh+a : h=0,1,\dots \} .$$
Suppose that  $n\notin \mathcal{U}$.  Then $n$ can be represented
as $p+2^k$, $p\in \mathcal{P}$, $k\in \mathbb{N}$. Then $$a\equiv
n\equiv p+2^k\pmod{q}.$$ Let $k'\equiv k\pmod{r}$ with $0\le k'\le
r-1$. Noting that $2^r\equiv 1\pmod q$, we have
\begin{equation}\label{g2}a\equiv n\equiv p+2^{k'}\pmod{q}.\end{equation} By $2\nmid
a$, we have $p\ge 3$. In view of $(p, m)=(n-2^k, m)=(a-2^k, m)>1$,
we have $p\mid m$.
 Since
$$1\le a\le \max\{ a, m\}<2^{r-2}<q,$$ $$ 1<p+2^{k'}\le
m+2^{r-1}\le 2^{r-2}+2^{r-1}<q,$$ it follows from \eqref{g2} that
$a=p+2^{k'}$. Noting that $a$ and $p$ are both odd, we have $k'\ge
1$, a contradiction with
$$a\in \bigcup_{i\in I} A_i\subseteq \mathcal{U}.$$
Hence,  $n\in \mathcal{U}$. It follows that
$$\{ mqh+a : h=0,1,\dots \} \subseteq \mathcal{U}.$$
This means that $$a\in \bigcup_{i\in K} A_i.$$ So
$$\bigcup_{i\in I} A_i\subseteq \bigcup_{i\in K}
A_i.$$ Since $$\bigcup_{i\in K} A_i\subseteq \bigcup_{i\in I}
A_i,$$ it follows that
$$\bigcup_{i\in I} A_i=\bigcup_{i\in K} A_i.$$

This completes the proof of Theorem \ref{thm3}.\end{proof}

\section*{Acknowledgments}

This work was supported by the National Natural Science Foundation
of China, Grant No. 12171243  and a project funded by the Priority
Academic Program Development of Jiangsu Higher Education
Institutions..

\end{document}